\documentclass[letter]{article}
\usepackage[utf8]{inputenc}

\usepackage{graphicx}
\usepackage[]{geometry}

\usepackage{amsmath,amsthm,amsfonts,amssymb,amscd, dsfont}
\usepackage{lastpage}
\usepackage{enumerate}
\usepackage{fancyhdr}
\usepackage{mathrsfs}
\usepackage{xcolor}
\usepackage{listings}
\usepackage{hyperref}
\usepackage{imakeidx}
\usepackage{multirow}
\usepackage{float}
\usepackage{MnSymbol}
\usepackage{wasysym}
\usepackage{tikz}
\usepackage{tikz-cd}
\usepackage[T1]{fontenc}

\usepackage{authblk}
\usepackage{dirtytalk}

\newcommand{\mylabel}[2]{#2\def\@currentlabel{#2}\label{#1}}

\newcommand{\conv}[2]{\mathop{\underset{#2}{\overset{#1}\longrightarrow}}}

\newcommand{\maths}[1]{\[\begin{split}{#1}\end{split}\]}

\newcommand{\inv}{^{-1}}

\newcommand{\maps}[1]{\mathop{\overset{#1}\longrightarrow}}

\newcommand{\con}[2]{\{  {#1}\;|\; {#2} \} }
\newcommand{\Con}[2]{\big\{  {#1}\;\big|\; {#2} \big\} }
\newcommand{\COn}[2]{\Big\{  {#1}\;\Big|\; {#2} \Big\} }

\newcommand{\ku}{\mathcal}

\newtheorem{definition}{{Definition}}
\newtheorem{theorem}{{Theorem}}
\newtheorem{corollary}{{Corollary}}

\newtheorem{proposition}{{Proposition}}

\newtheorem{lemma}{{Lemma}}
\newtheorem{question}{{Question}}
\newtheorem{example}{{Example}}

\theoremstyle{definition}

\title{Coarse Structures on Homogeneous Spaces}

\author{Carlos P\'erez Estrada and Christian Rosendal} 

\date{}

\begin{document}

\maketitle

\begin{abstract}
Given a closed normal subgroup $H$ of a topological group $G$, we address the question of whether the left coarse structure on the quotient group $G/H$ equals the quotient of the left coarse structure on $G$.  We provide a counterexample among Polish groups, namely, the mapping class group of the Loch Ness monster surface seen as a quotient of the mapping class group of the punctured Loch Ness monster surface, and establish both equivalent and sufficient conditions for when this holds in special settings. The latter are formulated in terms of liftings of bounded sets, existence of transversals and  metrisability of the left coarse structure of $G$ restricted to $H$.
\end{abstract}

\section{Introduction}
The large-scale geometry of a topological group $G$ is naturally encoded by its left coarse structure $\ku E_L(G)$, which captures the behaviour of continuous isometric actions on metric spaces and extends the role played by word metrics in geometric group theory. This construction is functorial, in the sense that the operator $\ku E_L$ defines a functor between the categories of topological groups and coarse spaces, so not surprisingly a number of questions regarding the coarse structure of topological groups can be framed in terms of this functor. For example, it is known that it preserves direct products, that is, if $(G_i)_{i\in I}$ is a family of topological groups, then $\ku E_L\big(\prod_{i\in I}G_i\big)=\prod_{i\in I}\ku E_L(G_i)$ \cite[Section 3.5]{MR4327092}. Our paper deals with the specific question of whether $\ku E_L$ also preserves quotients and the various ramifications of this.

Although the motivating question is most simply stated for quotient groups, a substantial part of the paper is formulated for homogeneous spaces. Thus, if $H$ is a closed subgroup of $G$, not necessarily normal, the coset space $H\backslash G$ carries two natural coarse structures. One is the quotient coarse structure $\ku E_q(H\backslash G)$ obtained from $\ku E_L(G)$, while the other is the Hausdorff-distance coarse structure $\ku E_{\rm Haus}(H\backslash G)$ obtained by first pushing each continuous left-invariant pseudometric on $G$ to the corresponding Hausdorff pseudometric on $H\backslash G$ and then intersecting. When $H$ is normal, this second structure is exactly the usual left coarse structure on the quotient group $G/H$. In this sense the quotient-group problem below is the normal-subgroup instance of a more general question about whether the quotient and Hausdorff coarse structures on a homogeneous space agree.

\begin{question}\label{quest:intro1}
Let $H$ be a closed normal subgroup of a topological group $G$. Is the left coarse structure on the quotient $G/H$ equal to the quotient of the left coarse structure on $G$?
\end{question}
Whereas the analogous question for the left uniform structures has a positive answer, in coarse geometry the situation is far more subtle. Indeed, the answer to Question \ref{quest:intro1} is no in general and a counterexample arising among mapping class groups of infinite-type surfaces will be provided shortly. Remarkably, this single example simultaneously yields negative answers to several natural questions concerning quotient coarse structures, lifting of bounded sets, and preservation of metrisability. Still, the negative examples are counterbalanced by several positive results, whose common theme is that equality of coarse structures is governed by the possibility of lifting bounded sets in a controlled way.
\begin{question}\label{quest:intro2}
Let $H$ be a closed normal subgroup of a topological group $G$. Does every bounded subset of $G/H$ lift to a bounded subset of $G$?
\end{question}
By Corollary \ref{cor:lifting crit}, Questions \ref{quest:intro1} and \ref{quest:intro2} are equivalent for any specific pair of groups $G$ and $H$, but once we ask for a function lifting bounded sets to bounded sets, a so-called modest section for the quotient map, things become more complicated.

As is familiar from uniform spaces, metrisable coarse spaces are much easier to comprehend than general ones. In the setting of uniform spaces, the difference is mostly a matter of simple inconvenience as every uniformity is the union of a family of metrisable uniformities. However, for coarse structures the difference is more profound as not every coarse structure is the  intersection of metrisable coarse structures. Fortunately, for a large class of topological groups $G$, in particular, Polish groups, metrisability of $\ku E_L(G)$ is simply equivalent to $G$ being {\em locally bounded}, i.e., having a bounded identity neighbourhood \cite[Theorem 2.38]{MR4327092}. Let us also remark that, if $\ku E_L(G)$ is metrisable, Questions \ref{quest:intro1} and \ref{quest:intro2} have positive answers. This makes the following extension problem especially natural: can metrisability of the two end terms in a short exact sequence force metrisability in the middle, and thereby force the quotient map to behave well from the coarse point of view?
Metrisability of the left coarse structure has been widely studied specifically, for (pure) mapping class groups of infinite-type surfaces and infinite locally-finite graphs, for example \cite{MR4634747}, \cite{MR4892635} and \cite{MR4530621}. On the other hand, little is known about how these equivalent properties behave under group extensions or in short exact sequences.
\begin{question}\label{quest:intro3}
Let $G$ be a Polish group and $H$ a closed normal subgroup such that both $\ku E_L(H)$ and $\ku E_L(G/H)$ are metrisable. Does it follow that also $\ku E_L(G)$ is metrisable?
\end{question}

As noted above, there is a single quotient group providing a negative answer to each of Questions \ref{quest:intro1}, \ref{quest:intro2} and \ref{quest:intro3}. To see this, consider the \textit{generalized Birman exact sequence} applied to the mapping class group of the Loch Ness monster surface LNM,
\begin{equation}\label{BirmanLNM}
1\longrightarrow\pi_{1}\big(\mathrm{F}_{n}(\mathrm{LNM})\big)
{\longrightarrow} 
\mathrm{PMap}(\mathrm{LNM}\!\setminus\! \{n \mbox{ pts}\}) 
\maps{ } \mathrm{PMap}(\mathrm{LNM}) \longrightarrow 1,
\end{equation}
where $n\geqslant 1$ and 
$$
\mathrm{F}_{n}(\mathrm{LNM})=\Con{(x_{1},\ldots,x_n)\in \mathrm{LNM}^{n}}{x_{i}\neq x_{j}\mbox{ if }i\neq j}
$$ 
is the $n$-\textit{configuration space} of the Loch Ness monster surface. Equipping the mapping class groups above with their canonical Polish group topologies, we see that the quotient map 
$$
\mathrm{PMap}(\mathrm{LNM}\setminus \{n \mbox{ pts}\}) {\longrightarrow} \mathrm{PMap}(\mathrm{LNM})
$$ 
is a continuous epimorphism and that the kernel $H=\pi_{1}\big(\mathrm{F}_{n}(\mathrm{LNM})\big)$ is a closed normal  subgroup of  $G=\mathrm{PMap}(\mathrm{LNM}\setminus \{n \mbox{ pts}\})$. Because the $n$-configuration space of a manifold is itself a manifold, the group $\pi_{1}\big(\mathrm{F}_{n}(\mathrm{LNM})\big)$ is countable discrete and has metrisable left coarse structure. Similarly, the left coarse structure of the quotient group $\mathrm{PMap}(\mathrm{LNM})$ can be metrised with a bounded metric. On the other hand, the left coarse structure of the middle group $\mathrm{PMap}(\mathrm{LNM}\setminus\{n\mbox{ pts}\})$ is non-metrisable by \cite[Theorem 1.1]{MR4892635} and \cite[Theorem 1.4]{MR4634747}. In other words, the short exact sequence above is a counterexample to Question \ref{quest:intro3}. To see that it is also a counterexample to Questions \ref{quest:intro1} and \ref{quest:intro2}, we rely on our first main result.

\begin{theorem}\label{thm:intro1}
Suppose $H$ is a closed normal subgroup of a Polish group $G$ and assume that the quotient group $G/H$ has metrisable coarse structure. Then the following are equivalent.
\begin{enumerate}
\item The left coarse structure on  $G/H$ equals the quotient of the left coarse structure on $G$,
\item every $\mathcal E_L(G/H)$-bounded set $B$ is of the form $\pi[A]$ for some bounded $A\subseteq G$,
\item $H$ is  {\em $\sigma$-cobounded} in $G$, that is, $G=\bigcup_{n=1}^\infty HA_n$ for a sequence of bounded subsets $A_n\subseteq G$, 
\item there is a modest section $G/H\maps\phi G$ for the quotient map $\pi$.
\end{enumerate}
\end{theorem}
Here, a section $G/H\maps\phi G$ for the quotient map $G\maps\pi G/H$ is said to be {\em modest} if $\phi[B]$ is bounded in $G$ whenever $B$ is $\ku E_L(G/H)$-bounded. Thus Theorem \ref{thm:intro1} says that, under the mild metrisability hypothesis on the quotient, all the natural ways in which a quotient map might preserve large-scale information are equivalent: equality of coarse structures, lifting of individual bounded sets, a countable covering of $G$ by bounded pieces modulo $H$, and the existence of one global bounded-set-preserving choice of representatives. In the Birman example above, $H$ and $G/H$ both have metrisable left coarse structure, while $G$ does not. By \cite[Corollary 4.16]{MR4327092} this precludes the existence of a modest section for the quotient map and hence we have the counterexample to Questions \ref{quest:intro1} and \ref{quest:intro2}.

The Birman sequence is the main geometric counterexample, but the paper also records a simpler linear one. In the final section we reinterpret a standard example of Bonet and Dierolf \cite[Example~1]{BonetDierolf1993} as a counterexample among abelian Polish groups: a continuous epimorphism of additive Polish groups has a bounded set in the quotient that cannot be lifted to a bounded set upstairs. This Bonet--Dierolf example isolates the lifting obstruction in a classical Fr\'echet-space setting, whereas the Birman example shows that the same obstruction can occur in mapping class groups and, moreover, can be tied to the failure of metrisability in the middle term of a short exact sequence.

The homogeneous-space viewpoint is also useful away from normal subgroups. In particular, for the diagonal subgroup $\Delta\leqslant G\times G$, the homogeneous space $\Delta\backslash(G\times G)$ is identified with $G$ by $\Theta(\Delta(g,f))=g\inv f$, and the quotient coarse structure on $\Delta\backslash(G\times G)$ corresponds precisely to the Roelcke coarse structure $\ku E_\wedge(G)$. The corresponding Hausdorff coarse structure is identified with $\bigcap_d\ku E_{d_\wedge}(G)$, so the same quotient-versus-Hausdorff comparison explains a natural question about the Roelcke coarse structure. For compact connected orientable surfaces this becomes quite concrete. On ${\sf Homeo}_0(S)$ we define a metric by comparing graphs of lifts in the universal cover and allowing admissible graph identifications; the resulting metric generates the Roelcke coarse structure. Thus two surface homeomorphisms are close in the Roelcke coarse geometry exactly when, after passing to suitable lifts, their graphs can be matched in the universal cover by uniformly bounded displacement.

We now turn from counterexamples and interpretations to positive criteria. These results should be read as complementary mechanisms for producing bounded lifts. Theorem \ref{thm:intro1} is a recognition theorem: in the locally bounded quotient case, any one of several bounded-lifting formulations implies all the others. The next result gives a concrete way to obtain such a lifting, namely by finding a transversal whose failure to be multiplicative is bounded inside the kernel.

\begin{theorem}\label{thm:intro2}
Let $H$ be a closed normal subgroup of a Polish group $G$. Suppose there is an  analytic transversal $T\subseteq G$ of the quotient map $G\maps \pi G/H$ for which the intersection $(T^{-2}T)\cap H$ is bounded in $G$. Then the section $G/H\maps \phi G$ for $\pi$ defined by $T$ is Borel measurable and modest with respect to the left coarse structures on $G/H$ and $G$. It follows that 
$$
\mathcal E_q(G/H)=\mathcal E_L(G/H)=\mathcal E_{\rm Haus}(G/H).
$$
\end{theorem}

Theorem \ref{thm:intro2} is most useful when one can choose representatives with bounded defect. Our final criterion goes in a different direction. Instead of constructing a section, it controls the geometry along the fibres of the quotient map. If the coarse structure induced on $H$ from $G$ is already captured by a single continuous left-invariant pseudometric on $G$, then boundedness in the quotient and boundedness along the fibres combine to produce bounded lifts.

\begin{theorem}\label{thm:intro3}
Let $H$ be a closed normal subgroup of a topological group $G$. Assume that there is a  continuous left-invariant pseudometric $\partial$ on $G$ such that $\partial\upharpoonright H$ is a compatible metric for the restriction of $\ku E_L(G)$ to $H$.
Then the left coarse structure on  $G/H$ equals the quotient of the left coarse structure on $G$.
\end{theorem}

This last criterion is first proved in the broader homogeneous-space setting as a bounded-lifting statement for $H\backslash G$, and then specialised to normal subgroups to obtain equality on $G/H$. We also give a product version for countable products of locally compact, $\sigma$-compact groups, where metrisability of the coarse structure induced on $H$ suffices to apply the same strategy.


\subsection*{Declaration on the use of AI-assisted tools}

AI-assisted tools, specifically OpenAI's ChatGPT/Codex tools, were used during the preparation of this manu\-script for auxiliary purposes. Namely, by identifying potentially relevant results from the literature, improving the clarity and organisation of the exposition, and proofreading the text. The tools were not used to generate mathematical results or as a substitute for verification. The authors take full responsibility for the content of the paper.

\subsection*{Acknowledgment of support}
The authors recognise the support of the US National Science Foundation via grant DMS-2246986. The first author also acknowledges financial support from the National Autonomous University of Mexico (UNAM), through the \say{Programa de Apoyo para la Movilidad en Posgrado} and PAPIIT Grant IN100326; from the Secretar\'ia de Ciencia, Humanidades, Tecnolog\'ia e Innovaci\'on (SECIHTI), through the \say{Convocatoria de apoyos complementarios de movilidad en el extranjero, movilidad nacional, movilidad en los sectores de inter\'es y movilidad para programas de doble titulaci\'on 2025} and Scholarship 1027866; and from Agros Horticultores S.A de C.V.


\section{Background on coarse spaces and their quotients}
Let us recall the definition of coarse structures and spaces due to J. Roe \cite{MR2007488}.
\begin{definition}
    A {\em coarse structure} on a set $X$ is an ideal $\mathcal E$ of subsets $E\subseteq X\times X$ called {\em entourages} such that
    \maths{
    \Delta&=\Con{(x,x)}{x\in X}\in \mathcal E,\\
    E^{\sf T}&=\Con{(y,x)}{(x,y)\in E}\in \mathcal E,\\
    E\circ F&=\Con{(x,z)}{(x,y)\in E\;\&\; (y,z)\in F \;\text{ for some }y\in X}\in \mathcal E
    }
    whenever $E,F\in \mathcal E$. A {\em coarse space} is a set $X$ equipped with a coarse structure $\mathcal E$.
\end{definition}

The notion of coarse structure is somewhat dual to uniform structures, which are filters as opposed to ideals. Thus, the appropriate maps between coarse spaces are similarly defined dually to uniformly continuous maps. Namely, a map $X\maps \phi Y$ between two coarse spaces $(X,\mathcal E)$ and $(Y,\mathcal F)$ is said to be {\em controlled} if
\maths{
\forall E\in \mathcal E\;\; \exists F\in \mathcal F\;\; \Big((x,x')\in E\Rightarrow (\phi(x),\phi(x'))\in F\Big).
}
Put differently, $\phi$ is controlled if $(\phi\times \phi)[E]\in \mathcal F$ for all entourages $E\in \mathcal E$.

\begin{example}
If $d$ is a pseudometric on $X$, we let 
$$
\mathcal E_d(X)=\COn{E\subseteq X\times X}{\sup_{(x,y)\in E}d(x,y)<\infty}
$$
be the coarse structure induced by $d$. Similarly, 
$$
\mathcal U_d(X)=\COn{E\subseteq X\times X}{\inf_{(x,y)\notin E}d(x,y)>0}
$$
is the uniform structure induced by $d$.
\end{example}

As $\mathcal P(X\times X)$ is a coarse structure and an arbitrary intersection of coarse structures is a coarse structure itself, it is immediate that any collection $\mathcal C$ of subsets of $X\times X$ generates a unique coarse structure, namely the smallest coarse structure containing $\mathcal C$.

\begin{definition}
Suppose $X$ is a coarse space and $\sim$ is an equivalence relation on $X$. The {\em quotient coarse structure} on $X/\!\sim$ is the smallest coarse structure for which the quotient map 
$$
X\overset{\pi}{\longrightarrow}X/\!\sim
$$
is controlled. Similarly, if $X$ is a uniform space with an equivalence relation $\sim$, the {\em quotient uniformity} on $X/\!\sim$ is the largest uniformity that renders $\pi$ uniformly continuous.
\end{definition}
In other words, the quotient coarse structure on $X/\!\sim$ is that generated by  the collection of all images $(\pi\times \pi)[E]$ of coarse entourages $E$ on $X\times X$.
If $\mathcal E$ denotes the coarse structure on $X$, we write $\mathcal E_q(X/\!\sim\!)$  to denote the quotient coarse structure on $X/\!\sim$ and $\mathcal U_q(X/\!\sim\!)$ to denote the quotient uniformity.

To see that this is the categorically correct definition for coarse structures, just note that, if $X\overset{\phi}{\longrightarrow} Y$ is any controlled map into another coarse space $Y$ so that $\phi(x)=\phi(x')$ whenever $x\sim x'$, then $\phi$ factors through $\pi$, that is, there is a unique controlled map $\big(X/\!\sim\!\big)\overset{\psi}{\longrightarrow} Y$ so that $\phi=\psi\circ \pi$. Indeed, the existence and uniqueness of the map $\psi$ is evident and to see that $\psi$ is controlled it suffices to see that, for any generator $F=(\pi\times \pi)[E]$ for the coarse structure on $X/\!\sim$, the image
$$
(\psi\times\psi)[F]= (\psi\times\psi)\big[(\pi\times\pi)[E]\big]=(\psi\pi\times\psi\pi)[E]=(\phi\times\phi)[E]
$$
is a coarse entourage on $Y$.

Just as a uniform structure on a set $X$ induces a topology on $X$, a coarse structure $\mathcal E$ on $X$ gives rise to the class of bounded subsets of $X$. Namely, $A\subseteq X$ is {\em bounded} if $A\times A\in \mathcal E$. Assuming furthermore that $(X,\mathcal E)$ is {\em connected}, meaning that $\bigcup\mathcal E=X\times X$ or, equivalently that $\{(x,y)\}\in \mathcal E$ for all $x,y\in X$, then the class of bounded sets is closed under finite unions. It follows that, if $(X,\mathcal E)$ is a connected coarse space, then the class $\mathcal B$ of bounded sets is a {\em bornology} on $X$, that is, an ideal of subsets such that $\bigcup \mathcal B=X$.

\begin{definition}
    A coarse structure $\mathcal E$ on a group $G$ is said to be {\em left-invariant}  if
    $$
    \widehat E=\Con{(gf,gh)}{(f,h)\in E\;\&\; g\in G }\in \mathcal E
    $$
    for all $E\in \mathcal E$. A bornology $\mathcal B$ on $G$ is said to be a {\em group bornology} if 
    $$
    A\inv\in \mathcal B\quad\&\quad  AB\in \mathcal B
    $$
    whenever $A,B\in \mathcal B$.\footnote{In \cite{MR2948279} left-invariant coarse structures are called {\em compatible} and group bornologies are called {\em generating families}.}
\end{definition}
As established in \cite{MR2948279}, for every group $G$, there is a bijective correspondence between left-invariant connected coarse structures and group bornologies on $G$. To define this, we set
$$
L_A=\Con{(g,f)}{g\inv f\in A}
$$
for any subset $A\subseteq G$.
The correspondence is then given by
\maths{
\mathcal E&\;\mapsto\; \mathcal B=\Con{A\subseteq G}{A\times A\in \mathcal E}\\
\mathcal B&\;\mapsto\; \mathcal E=\COn{E\subseteq G\times G}{E\subseteq L_A \text{ for some }A\in \mathcal B}.
}

\begin{lemma}\label{ConcreteBasisQuotientCoarseStructure}
Let $G$ be a group equipped with a left-invariant connected coarse structure $\mathcal E$ and let $H$ be a subgroup of $G$. Then a set $E\subseteq H\backslash G\times H\backslash G$ is an entourage for the quotient coarse structure on  $H\backslash G$ if and only if 
    $$
    E\;\subseteq \;L_A^{H\backslash G}\colon=\COn{(Hg,Hf)}{g\inv Hf\cap A\neq \emptyset}
    $$
for some  bounded subset $A\subseteq G$.
\end{lemma}

\begin{proof}
Because $\mathcal E$ is connected and left-invariant, it is generated by the entourages $L_A$ for $A$ bounded in $G$. It thus follows that the quotient coarse structure $\mathcal E_q(H\backslash G)$ is generated by their images
$(\pi\times \pi)[L_A]$ where $G\maps \pi H\backslash G$ denotes the quotient map, $\pi(g)=Hg$. Observe now that
\maths{
(Hg,Hf)\in (\pi\times \pi)[L_A ]
& \;\Leftrightarrow\; \exists h_1,h_2\in H\,\colon\; (h_1g,h_2f)\in L_A\\
& \;\Leftrightarrow\;  \exists h_1,h_2\in H\,\colon\;  g\inv h_1\inv h_2f\in A\\
& \;\Leftrightarrow\; g\inv Hf\cap A\neq \emptyset\\
& \;\Leftrightarrow\; (Hg,Hf)\in L^{H\backslash G}_A,
}
i.e., $L^{H\backslash G}_A=(\pi\times \pi)[L_A ]$. Also, it is immediate that
$$
\big(L^{H\backslash G}_A\big)^{\sf T}=L^{H\backslash G}_{A\inv}
\qquad\&\qquad
L^{H\backslash G}_A \cup L^{H\backslash G}_B=L^{H\backslash G}_{A\cup B} 
\qquad\&\qquad 
L^{H\backslash G}_A\circ L^{H\backslash G}_B=L^{H\backslash G}_{AB}.
$$
Therefore, the coarse structure generated by the sets $L^{H\backslash G}_A=(\pi\times \pi)[L_A ]$ is  exactly the collection of sets $E\subseteq H\backslash G\times H\backslash G$ contained in one of them.
\end{proof}

\begin{lemma}\label{BoundedSetsQuotient}
Let $G$ be a group equipped with a left-invariant connected coarse structure $\mathcal E$, let $H$ be a subgroup of $G$ and $G\maps\pi H\backslash G$ be the quotient map. Then a set $B\subseteq H\backslash G$ is bounded in the quotient coarse structure if and only if 
 $$
 B=\pi[A]
 $$
for some  bounded subset $A\subseteq G$.
\end{lemma}

\begin{proof}
Because the quotient map is bounded, it maps a bounded set $A$ to a bounded set $\pi[A]$. For the reverse implication, suppose $\emptyset \neq B\subseteq H\backslash G$ is bounded and find some bounded set $C\subseteq G$ so that 
$$
B\times B\;\;\subseteq\;\; L^{H\backslash G}_C\;=\;(\pi\times \pi)[L_C].
$$
Fix any $Hg\in B$. Note that, if $Hf\in B$, then $(Hg,Hf)\in B\times B$ and so $g\inv hf\in C$ and $hf\in gC$ for some $h\in H$, whence $Hf=\pi(hf)\in \pi[gC]$. Setting $A=gC\cap \pi\inv(B)$, which is bounded, we have that $B=\pi[A]$.
\end{proof}

When $H$ is a normal subgroup of a group $G$, we will use the more common notation $G/H$ for the quotient group rather than 
$H\backslash G$. Of course, this makes no difference and the quotient map $G\maps\pi G/H$ is formally the same function, $\pi(g)=gH=Hg$. But the notation $G/H$ does remind us that we are dealing with a group and not just a homogeneous space.

\begin{lemma}\label{QuotientLeftInvariant}
Let $G$ be a group with a left-invariant connected coarse structure $\mathcal E$ and  a normal subgroup $H$ with quotient map $G\maps\pi G/H$. Then  $\mathcal E_q(G/H)$ is left-invariant and is therefore  determined by its collection of bounded sets, which are the sets of the form $\pi[A]$ for $A\subseteq G$ bounded.
\end{lemma}

\begin{proof}
For any $g,f,k\in G$ and bounded set $A\subseteq G$, we have
\maths{
(gH,fH)\in L^{G/H}_A
&\quad\Leftrightarrow\quad g\inv H f\cap A\neq \emptyset\\
&\quad\Leftrightarrow\quad g\inv k\inv H kf\cap A\neq \emptyset\\
&\quad\Leftrightarrow\quad (kgH,kfH)\in L^{G/H}_A,
}
which shows that $\widehat{L^{G/H}_A}=L^{G/H}_A\in \mathcal E_q(G/H)$ and thus that $\mathcal E_q(G/H)$ is left-invariant.
\end{proof}

Suppose $d$ is a {\em left-invariant} pseudometric on a group $G$, that is, 
$$
d(hg,hf)=d(g,f)
$$
for all $g,f,h\in G$. Then the coarse structure $\mathcal E_d(G)$ induced by $d$ is evidently left-invariant. Assume furthermore that  $H$ is a $d$-closed subgroup and consider the induced Hausdorff distance $d_{\mathrm{Haus}}$ on $H\backslash G$, namely,
\begin{align*}
    d_{\mathrm{Haus}}(Hg,Hf)
    \;=\;\max\bigg\{\sup_{h\in H}\inf_{k\in H}d(hg,kf),\sup_{k\in H}\inf_{h\in H}d(hg,kf)\bigg\}.
 \end{align*}
Because the right cosets $Hg$ may have infinite diameter, a priori, $d_{\mathrm{Haus}}$ could take infinite values, but we shall see that this is not the case.

 \begin{proposition}\label{prop:individual metrics}
Suppose  $d$ is a left-invariant pseudometric on a group $G$ and $H$ is a $d$-closed subgroup. Then, for all $Hg,Hf\in H\backslash G$, we have
$$
d_{\mathrm{Haus}}(Hg,Hf)\;=\; \inf_{h\in H}d(hg,f)
$$
Moreover, $\mathcal E_{d_{\rm Haus}}(H\backslash G)$ and $\mathcal U_{d_{\rm Haus}}(H\backslash G)$ equal the quotients of respectively $\mathcal E_{d}(G)$ and $\mathcal U_{d}(G)$.  
\end{proposition}

\begin{proof}
The equality $d_{\mathrm{Haus}}(Hg,Hf)\;=\; \inf_{h\in H}d(hg,f)$ is immediate from the left-invariance of $d$. In particular, $d_{\mathrm{Haus}}$ is finite-valued even though the cosets may have infinite diameter and so $d_{\rm Haus}$ is a pseudometric on $H\backslash G$.

Let  $\big(\mathcal E_{d}\big)_q(H\backslash G)$ and  $\big(\mathcal U_{d}\big)_q(H\backslash G)$ denote the quotients of  $\mathcal E_d(G)$ and  $\mathcal U_d(G)$. For $r>0$, let also $B_d(r)=\con{g\in G}{d(g,1)<r}$. 
By Lemma \ref{ConcreteBasisQuotientCoarseStructure}, we have that, for all $E\subseteq H\backslash G\times H\backslash G$, 
\maths{
E\in \big(\mathcal E_{d}\big)_q(H\backslash G)
&\quad\Leftrightarrow\quad \exists r \,\colon\;  E\subseteq L^{H\backslash G}_{B_d(r)}\\
&\quad\Leftrightarrow\quad \exists r \;\forall (Hg,Hf)\in E\,\colon \;   g\inv Hf\cap B_d(r)\neq\emptyset\\
&\quad\Leftrightarrow\quad \sup_{ (Hg,Hf)\in E}\; d_{\rm Haus}(Hg,Hf)=\sup_{ (Hg,Hf)\in E}\;\inf_{h\in H}d(1, g\inv hf)<\infty\\
&\quad\Leftrightarrow\quad E\in \mathcal E_{d_{\rm Haus}}(H\backslash G),
}
showing that $ \big(\mathcal E_{d}\big)_q(H\backslash G)=\mathcal E_{d_{\rm Haus}}(H\backslash G)$.

To establish the equality $ \big(\mathcal U_{d}\big)_q(H\backslash G)=\mathcal U_{d_{\rm Haus}}(H\backslash G)$, let us first note that the quotient map $(G,d)\maps\pi \big(H\backslash G,d_{\rm Haus}\big)$ is Lipschitz and hence uniformly continuous. Because the quotient uniformity is the largest for which the quotient map from $(G,d)$ is uniformly continuous, it follows that 
$\mathcal U_{d_{\rm Haus}}(H\backslash G)\subseteq  \big(\mathcal U_{d}\big)_q(H\backslash G)$. Conversely, suppose $E\in  \big(\mathcal U_{d}\big)_q(H\backslash G)$. Then $(\pi\times \pi)\inv(E)\in \mathcal U_d(G)$ and so there is some $\epsilon>0$ so that $\big(\pi(g),\pi(f)\big)\in E$ whenever $d(g,f)<\epsilon$.  Therefore, if $d_{\rm Haus}(Hg,Hf)<\epsilon$, we have $d(hg,f)<\epsilon$ for some $h\in H$ and hence $(Hg,Hf)=\big(\pi(hg),\pi(f)\big)\in E$. It follows that $E\in \mathcal U_{d_{\rm Haus}}(H\backslash G)$, proving that  $ \big(\mathcal U_{d}\big)_q(H\backslash G)=\mathcal U_{d_{\rm Haus}}(H\backslash G)$.
\end{proof}


\section{The coarse structure of a topological group}
Assume now that $G$ is a topological group. The {\em left-uniform structure} $\mathcal U_L(G)$ on $G$ is that generated by all entourages 
$$
L_V=\Con{(g,f)\in G\times G}{g\inv f\in V},
$$
where $V$ ranges over identity neighbourhoods in $G$. It is well-known that this can alternatively be described as the union
$$
\mathcal U_L(G)=\bigcup_d\mathcal U_d(G)
$$
of the uniformities $\mathcal U_d(G)$ induced by the collection of all continuous {\em left-invariant} pseudometrics $d$ on $G$, i.e., such that, for all $g,f,h\in  G$, 
$$
d(hg,hf)=d(g,f).
$$
Thus, for $E\subseteq G\times G$,
\begin{equation}
E\in \mathcal U_L(G) \quad\Leftrightarrow\quad \exists d\; \inf_{(g,f)\notin E}d(g,f)>0.
\end{equation}

In analogy with the left-uniform structure, we may define a left coarse structure on $G$.
\begin{definition}\cite{MR4327092}
    Let $G$ be a topological group. The {\em left coarse structure} $\mathcal E_L(G)$ is given by
$$
\mathcal E_L(G)=\bigcap_d\mathcal E_d(G),
$$
where again $d$ ranges over  continuous  left-invariant pseudometrics.
\end{definition}
In other words,
\begin{equation}
E\in \mathcal E_L(G) \quad\Leftrightarrow\quad \forall d\; \sup_{(g,f)\in E}d(g,f)<\infty.
\end{equation}
In particular, for any subset $A\subseteq G$, we find that $A$ is $\ku E_L(G)$-bounded if and only if 
$$
A \text{ is $\ku E_L(G)$-bounded }\quad\Leftrightarrow\quad
\forall d\quad {\sf diam}_d(A)<\infty.
$$
Because $\ku E_L(G)$ is left-invariant, this in turn shows that
$$
E\in \ku E_L(G)\quad\Leftrightarrow\quad E\subseteq L_A \text{ for some $\ku E_L(G)$-bounded set }A.
$$

 \begin{lemma}
Suppose  $d$ is a continuous left-invariant pseudometric on a topological group $G$ and $H$ is a closed subgroup of $G$. Then the induced Hausdorff distance  $d_{\rm Haus}$ is a continuous pseudometric on $H\backslash G$. Similarly, if $d$ is a compatible left-invariant metric on $G$, then $d_{\rm Haus}$ is a compatible metric on $H\backslash G$.\end{lemma}

\begin{proof}
Assume first that  $d$ is a continuous left-invariant pseudometric. Because the quotient map $G\maps\pi H\backslash G$ is surjective, continuous and open, so is the product map $G\times G\maps{\pi\times \pi}H\backslash G\times H\backslash G$. Consequently, $H\backslash G\times H\backslash G\maps {d_{\rm Haus}} \mathbb R$ is continuous if and only if the composition $d_{\rm Haus}\circ(\pi\times \pi)$ is continuous. However, $d_{\rm Haus}\big(\pi(g),\pi(f)\big)=\inf_{h\in H}d(g,hf)$ and the latter expression is evidently continuous in $(g,f)$, so we conclude that $d_{\rm Haus}$ is continuous.

Assume now that $d$ is a compatible left-invariant metric on $G$ and  suppose $\inf_{h\in H}d(g,hf_i)=d_{\rm Haus}(Hg,Hf_i)\conv{}i0$ for some net $(Hf_i)_i$ and some $Hg$. This means that there are $h_i\in H$ so that $d(g, h_if_i)\conv{}i0$, whereby $h_if_i\conv{}ig$ and hence $Hf_i=\pi(h_if_i)\conv{}i\pi(g)=Hg$. Since $d_{\rm Haus}$ is already known to be continuous, this shows that 
$$
d_{\rm Haus}(Hg,Hf_i)\conv{}i0\;\Leftrightarrow\; Hf_i\conv{}iHg
$$
and thus that $d_{\rm Haus}$ is compatible with the topology of $H\backslash G$. That it is a metric follows from $H$ being closed in $G$.
\end{proof}

\begin{definition}
    Let $G$ be a topological group and $H$ a closed subgroup. The {\em Hausdorff distance} coarse structure on $H\backslash G$ is given by 
\maths{
    \mathcal{E}_{\mathrm{Haus}}(H\backslash G)
    = \bigcap_d\mathcal E_{d_{\rm Haus}}(H\backslash G),   
    }
    where $d$ ranges over all continuous left-invariant pseudometrics on $G$.
\end{definition}

By definition, the quotient coarse structure $\mathcal{E}_{q}(H\backslash G)$ is obtained by first taking the intersection $\mathcal{E}_{L}(G)=\bigcap_d\mathcal{E}_d(G)$ and then passing to the quotient. On the other hand, by Proposition \ref{prop:individual metrics}, $\mathcal{E}_{\mathrm{Haus}}(H\backslash G)$ is obtained by first taking the quotients of each of the coarse structures $\mathcal E_d(G)$ and subsequently taking the intersection of these. That this second procedure is not entirely frivolous is exemplified in the following.

\begin{proposition}\label{prop:quotient coarse}
Let $G$ be a topological group and $H$ a closed subgroup. Then 
$$
\mathcal E_q(H\backslash G)\;\subseteq\; \mathcal E_{\rm Haus}(H\backslash G).
$$
If furthermore $H$ is normal in $G$, then every continuous left-invariant pseudometric on $G/H$ has the form $d_{\rm Haus}$ for some continuous left-invariant pseudometric $d$ on $G$ and hence
     $$
     \mathcal E_L(G/H)=\mathcal E_{\rm Haus}(G/H).
     $$
\end{proposition}

\begin{proof}
    Observe that, if $G\maps\pi H\backslash G$ denotes the quotient map and $E$ is a coarse entourage on $G$, then
\maths{
\sup_{(g,f)\in E}d_{\text{Haus}}\big(\pi(g),\pi(f)\big)
\;=\;
\sup_{(g,f)\in E}d_{\text{Haus}}\big(Hg,Hf\big)
\;\leqslant \; \sup_{(g,f)\in E}d(g,f)
\;<\; \infty
}
for every continuous left-invariant pseudometric $d$ on $G$. It therefore follows that $(\pi\times \pi)[E]\in \mathcal{E}_{\mathrm{Haus}}(H\backslash G)$, showing that $\pi$ is a controlled map with respect to the Hausdorff distance coarse structure on $H\backslash G$. Because the quotient coarse structure is the smallest for which $\pi$ is controlled, this immediately shows that $\mathcal E_q(H\backslash G)\;\subseteq\; \mathcal E_{\rm Haus}(H\backslash G)$.

Suppose now that $H$ is furthermore normal in $G$. Observe that, if $d$ is a continuous left-invariant pseudometric on $G$, then $d_{\rm Haus}$ is a continuous left-invariant pseudometric on $G/H$. Conversely, if $D$ is any continuous left-invariant pseudometric on $G/H$, then $d=D\circ(\pi\times \pi)$ is continuous and left-invariant on $G$ and satisfies 
$D= d_{\rm Haus}$.
It thus follows that
$$
 \mathcal E_L(G/H) 
 \;=\;\bigcap_d\mathcal E_{d_{\rm Haus}}(G/H)
\;= \;  \mathcal E_{\rm Haus}(G/H)
$$
as claimed.
\end{proof}

When $H$ is a closed normal subgroup of $G$, there are two natural coarse structures on the quotient $G/H$:  Either take the  topological quotient group $G/H$ and then produce the associated left coarse structure, i.e., $\mathcal E_L(G/H)= \mathcal E_{\rm Haus}(G/H)$. Alternatively, we first produce the left coarse structure $\mathcal E_L(G)$ and then pass to the quotient $\mathcal E_q(G/H)$. This naturally raises the question whether these two approaches result in the same object. As we shall see, it does not in general, so instead we aim to understand when it does. We formulate this more broadly for the case when $H$ is no longer assumed to be normal in $G$.

\begin{question}\label{quest:1}
What are the conditions on  topological groups $G$ and closed subgroups $H$ that ensure that 
$\mathcal{E}_{q}(H\backslash G)=\mathcal{E}_{\mathrm{Haus}}(H\backslash G)$\;? 
\end{question}

Returning again to the case of normal subgroups, we can reformulate this in terms of liftings of bounded sets.

\begin{corollary}\label{cor:lifting crit}
Let $G$ be a topological group and $H$ a closed normal subgroup. Then 
$$
\mathcal E_q(G/H)=\mathcal E_L(G/H)=\mathcal E_{\rm Haus}(G/H)
$$ 
if and only if every $\mathcal E_{\rm Haus}(G/H)$-bounded set is of the form $\pi[A]$ for some bounded set $A\subseteq G$.
\end{corollary}

\begin{proof}
By Lemma \ref{QuotientLeftInvariant}, $\mathcal E_q(G/H)$ is completely determined by its collection of bounded sets, which are exactly the sets of the form $\pi[A]$ for $A\subseteq G$ bounded in $\mathcal E_L(G)$. Similarly, $\mathcal E_L(G/H)=\mathcal E_{\rm Haus}(G/H)$ is connected and left-invariant and therefore also determined by its bounded sets. Therefore,  $\mathcal E_q(G/H)=\mathcal E_L(G/H)=\mathcal E_{\rm Haus}(G/H)$ if and only if the coarse structures have the same bounded sets or, equivalently, if and only if every $\mathcal E_{\rm Haus}(G/H)$-bounded set is of the form $\pi[A]$ for some bounded set $A\subseteq G$.
\end{proof}

\begin{question}\label{quest:2}
  Let $G$ be a topological group and $H$ a closed  subgroup.  Is it true that $\mathcal{E}_{q}(H\backslash G)=\mathcal{E}_{\mathrm{Haus}}(H\backslash G)$  if and only if every  $\mathcal{E}_{\mathrm{Haus}}(H\backslash G)$-bounded subset is of the form $\pi[A]$ for some bounded set $A\subseteq G$?
\end{question}

Similar to the case of coarse structures, if $H$ is a closed subgroup of a topological group $G$, we let $\mathcal U_{\rm Haus}(H\backslash G)$ denote the uniform structure on $H\backslash G$  given by
$$
\mathcal U_{\rm Haus}(H\backslash G)=\bigcup_d\mathcal U_{d_{\rm Haus}}(H\backslash G),
$$
where $d$ ranges over continuous left-invariant pseudometrics on $G$. Contrary to the case of coarse structures, this construction always produces the quotient uniformity on $H\backslash G$.

\begin{proposition}
Let $G$ be a topological group and $H$ a closed subgroup. Then 
$$
\mathcal U_{\rm Haus}(H\backslash G)\;=\; \mathcal U_q(H\backslash G).
$$
If furthermore $H$ is normal in $G$, then
     $$
    \mathcal U_L(G/H)\;=\;\mathcal U_{\rm Haus}(G/H)\;=\; \mathcal U_q(G/H).
     $$
\end{proposition}

\begin{proof}
Observe first that, if $d$ is a continuous left-invariant pseudometric on $G$, then the quotient map $(G,d)\maps\pi(H\backslash G, d_{\rm Haus})$ is Lipschitz and thus uniformly continuous. Because $\mathcal U_L(G)=\bigcup_d\mathcal U_d(G)$, whereas $\mathcal U_{\rm Haus}(H\backslash G)=\bigcup_d\mathcal U_{d_{\rm Haus}}(H\backslash G)$, it follows that 
$$
\big(G,\mathcal U_L(G)\big)\maps\pi\big(H\backslash G, \mathcal U_{\rm Haus}(H\backslash G)\big)
$$
is uniformly continuous. But $\mathcal U_q(H\backslash G)$ is the largest uniform structure making $\pi$ uniformly continuous, so $\mathcal U_{\rm Haus}(H\backslash G)\subseteq \mathcal U_q(H\backslash G)$.

To establish the reverse inclusion, suppose $E\in \mathcal U_q(H\backslash G)$ is given. Because $\pi$ is uniformly continuous with respect to $\mathcal U_q(H\backslash G)$ and $\mathcal U_L(G)=\bigcup_d\mathcal U_d(G)$, there is some continuous left-invariant pseudometric $d$ and some $\epsilon>0$ so that
$$
d(g,f)<\epsilon\;\Rightarrow\; \big(\pi(g),\pi(f)\big)\in E,
$$
whereby
$$
d_{\rm Haus}\big(\pi(g),\pi(f)\big)<\epsilon\;\Rightarrow\; \big(\pi(g),\pi(f)\big)\in E.
$$
It follows that $E\in \mathcal U_{d_{\rm Haus}}(H\backslash G)\subseteq \mathcal U_{{\rm Haus}}(H\backslash G)$ as required.

Assume now that $H$ is also normal in $G$. Then 
$$
\mathcal U_{\rm Haus}(G/H)
\;=\;\bigcup_d\mathcal U_{d_{\rm Haus}}(G/H)
\;\subseteq\;
\bigcup_D\mathcal U_{D}(G/H)
\;=\;\mathcal U_L(G/H),
$$
where $D$ and $d$ vary over continuous left-invariant pseudometrics on $G/H$, respectively  on $G$. On the other hand, if $D$ is a continuous left-invariant pseudometric on $G/H$, then $d=D\circ(\pi\times \pi)$ is a continuous left-invariant pseudometric on $G$ for which $d_{\rm Haus}=D$. So the inclusion above is actually an equality.
\end{proof}



\section{The Roelcke coarse structure}
As is well-known, both the class of coarse structures and the class of uniformities on a given set $X$ form complete (non-distributive) lattices. Here, for two uniformities $\mathcal U$ and $\mathcal V$ on $X$, we set
$$
\ku U \leqslant \ku V \quad\Leftrightarrow \quad \ku U\subseteq \ku V,
$$
whereas for two coarse structures $\ku E$ and $\ku F$, we let 
$$
\ku E \leqslant \ku F \quad\Leftrightarrow \quad \ku E\supseteq \ku F.
$$
Because an arbitrary intersection of coarse structures is again a coarse structure, we find that
$$
\bigvee_{i\in I}\ku E_i=\bigcap_{i\in I}\ku E_i
$$
for any family $(\ku E_i)_{i\in I}$. On the other hand, $\bigwedge_{i\in I}\ku E_i$ is the unique coarse structure generated by the union $\bigcup_{i\in I}\ku E_i$ and is not, in general, the union itself.

On any group $G$, we have a duality between the left and right-invariant pseudometrics, which is given by the involution
$d\mapsto \check d$, where
$$
\check d(g,f)=d(g\inv, f\inv)
$$
for $g,f\in G$. To simplify the exposition, we shall henceforth assume that $d$ denotes a left-invariant pseudometric on $G$, whereas $\check d$ denotes its corresponding right-invariant dual.

For a topological group $G$, the {\em right-coarse} structure $\mathcal{E}_{R}(G)$ and {\em right-uniform} structure $\mathcal U_R(G)$ can be described similarly to their left analogues  by
$$
\mathcal E_R(G) \;=\;\bigcap_d\mathcal E_{\check d}\;=\;\bigvee_d\mathcal E_{\check d}
\qquad\qquad \&\qquad\qquad 
\mathcal U_R(G) \;=\;\bigcup_d\mathcal U_{\check d}\;=\; \bigvee_d\mathcal U_{\check d}
$$
where again $d$ varies over the continuous left-invariant pseudometrics on $G$.

The {\em Roelcke uniformity} on a topological group $G$ \cite{MR644485} is defined via
\begin{equation}\label{eq:Roelcke unif}
\mathcal{U}_{\wedge}(G) \;=\; \ku U_L(G)\wedge \ku U_R(G)\;=\;\Big( \bigvee_d\ku U_d(G)\big) \wedge \Big(\bigvee_d\ku U_{\check d}(G)\Big)
\end{equation}
and similarly, the {\em Roelcke-coarse structure} $\mathcal{E}_{\wedge}(G)$ introduced by J. Zielinski \cite{MR4349657} is defined by 
\begin{equation}\label{eq:Roelcke coarse}
\mathcal{E}_{\wedge}(G) \;=\; \ku E_L(G)\wedge \ku E_R(G)\;=\;\Big( \bigvee_d\ku E_d(G)\big) \wedge \Big(\bigvee_d\ku E_{\check d}(G)\Big).
\end{equation}
The definition above introduces $\ku E_R(G)$, and thus $\mathcal{E}_{\wedge}(G)$, via continuous right-invariant pseudometrics, but of course can also be described in terms of the group bornology of $\ku E_R(G)$-bounded sets. Let us first note that $\ku E_L(G)$ and $\ku E_R(G)$ have the same bounded sets. Indeed, for any $A\subseteq G$,
\maths{
A \text{ is $\ku E_L(G)$-bounded }
&\quad\Leftrightarrow \quad \forall d\;\;\; {\sf diam}_d(A)<\infty\\
&\quad\Leftrightarrow \quad \forall d\;\;\;\sup_{g\in A}\;\check d(1,g)=\sup_{g\in A}\;d(1, g\inv )=\sup_{g\in A}d(g,1)<\infty\\
&\quad\Leftrightarrow \quad \forall d\;\;\; {\sf diam}_{\check d}(A)<\infty\\
&\quad\Leftrightarrow \quad A \text{ is $\ku E_R(G)$-bounded}.\\
}
So henceforth, we shall simply say {\em bounded} to mean $\ku E_L(G)$ or, equivalently $\ku E_R(G)$-bounded.
For a subset $A\subseteq G$, set  
$$
R_A=\Con{(g,f)\in G\times G}{gf\inv \in A}.
$$
Then, exactly as for $\ku E_L(G)$, we find that
\begin{equation}\label{eq:Right coarse}
E\in \ku E_R(G) \quad\Leftrightarrow \quad E\subseteq R_A \text{ for some bounded set }A.
\end{equation}

Note now that, for $A,B\subseteq G$, we have
\maths{
(g,f)\in L_A\circ R_{B}
\;\Leftrightarrow\;   \exists h\;\big( g\inv h\in A\;\&\; hf\inv \in B\big) \;\Leftrightarrow\;  gA\cap B f\neq\emptyset
\;\Leftrightarrow\;  (g,f)\in R_{B}\circ L_{A}
}
and so $L_A\circ R_{B}=R_B\circ L_A$. In particular, 
$$
\big(L_A\circ R_A\big)\circ \big(L_B\circ R_B\big)=L_A\circ \big(R_A\circ L_B\big)\circ R_B=L_A\circ \big(L_B\circ R_A\big)\circ R_B= L_{AB}\circ R_{AB}
$$
and so the entourages $L_A\circ R_A$ with $A$ varying over bounded subsets of $G$ form a coarse structure on $G$.
Because $\ku E_L(G)$ is generated by the entourages $L_A$ and $\ku E_R(G)$ by the $R_B$ for $A$ and $B$ bounded and because $L_A\circ R_{B}=R_B\circ L_A$, we find, as was established in  \cite[Proposition 1]{MR4349657}, that $\ku E_\wedge(G)$ can equivalently be described by 
\begin{equation}\label{eq:Roelcke basis}
E\;\in\; \ku E_\wedge(G) \;\;\Leftrightarrow \;\; E\;\subseteq\; L_A\circ R_A  =\Con{(g,f)\in G\times G}{gA\cap A f\neq\emptyset}\text{ for some bounded set }A.
\end{equation}
It follows from this that,  if $B\subseteq G$ is bounded in the Roelcke coarse structure, we can find a $\ku E_L(G)$-bounded set $A$ so that $B\times \{1\}\subseteq L_A\circ R_A$, whereby $B\subseteq AA\inv$. Since also $AA\inv$ is $\ku E_L(G)$-bounded,we conclude that $B$ itself is $\ku E_L(G)$-bounded. Since the Roelcke coarse structure contains the left coarse structure, we conclude that, for any subset $B\subseteq G$, 
\maths{
B \text{ is $\ku E_L(G)$-bounded}
\quad \Leftrightarrow \quad B \text{ is $\ku E_R(G)$-bounded}
\quad \Leftrightarrow \quad B \text{ is $\ku E_\wedge(G)$-bounded}.
}
Again, this stresses the fact that the notion of bounded subsets of $G$ is really unambiguous.

If $d$ is a continuous left-invariant pseudometric on $G$, we define another continuous, though not necessarily left-invariant, pseudometric $d_\wedge$ by the formula
$$
d_{\wedge}(g,f)=\inf_{h\in G}\max\big\{d(g,h),\check d(h,f)\big\}.
$$
Because $d_\wedge\leqslant d, \check d$, we immediately see that $\ku E_d(G)\wedge \ku E_{\check d}(G)\subseteq \ku E_{d_\wedge}(G)$. Conversely, if $E\in \ku E_{d_\wedge}(G)$, we have $R=\sup_{(g,f)\in E}d_\wedge(g,f)<\infty$. Thus, if $(g,f)\in E$, there is some $h\in G$ for which $d(g,h)\leqslant R+1$ and $\check d(h,f)\leqslant R+1$, showing that
$$
E\;\subseteq\; \Con{(g,h)\in G\times G}{d(g,h)\leqslant R+1}\circ \Con{(h,f)\in G\times G}{\check d(h,f)\leqslant R+1}\;\in\; \ku E_d(G)\wedge \ku E_{\check d}(G).
$$
In other words, we have $\ku E_d(G)\wedge \ku E_{\check d}(G)=\ku E_{d_\wedge}(G)$ and therefore
$$
\bigcap_d\ku E_{d_\wedge}(G)=\bigvee_d\Big(\ku E_d(G)\wedge \ku E_{\check d}(G)\Big),
$$
where, as usual, the $d$ varies over all continuous left-invariant pseudometrics on $G$.

\begin{question}\label{RoelckeQuestion}
Characterize the topological groups $G$ for which
$$
\mathcal{E}_{\wedge}(G)\;=\;\Big( \bigvee_d\ku E_d(G)\big) \wedge \Big(\bigvee_d\ku E_{\check d}(G)\Big)
\;=\;\bigvee_d\Big(\ku E_d(G)\wedge \ku E_{\check d}(G)\Big)\;=\; \bigcap_d\ku E_{d_\wedge}(G).
$$
\end{question}
By our remarks above, only the middle equality is in question. The other two hold unconditionally.

The Roelcke uniformity introduced above admits an alternative description \cite{MR644485}. Namely, let $G$ be a topological group and let $\Delta=\con{(h,h)\in G\times G}{h\in G}$ denote the diagonal subgroup of $G\times G$. Observe that $\Delta$ is in general not normal in $G$.
The map $(g,f)\in G\times G\mapsto g\inv f\in G$ factors through the quotient $\Delta\backslash (G\times G)$, which produces a bijection
\maths{
\Delta\backslash (G\times G) &\maps \Theta G
}
given by $\Theta\big(\Delta(g,f)\big)=g\inv f$. If we equip $G\times G$ with its left-uniformity $\ku U_L(G\times G)$, which is nothing but the product uniformity $\ku U_L(G)\times \ku U_L(G)$, and $\Delta\backslash (G\times G)$ with the induced quotient uniformity $\ku U_q\big(\Delta\backslash (G\times G)\big)$, then $\Theta$ is a uniform homeomorphism when $G$ is given its Roelcke uniformity $\ku U_\wedge(G)$. We show that the same holds for the corresponding coarse structures too.

\begin{theorem}\label{thm:quotient-roelcke}
Let $G$ be a topological group. Then the map $\Delta\backslash (G\times G) \maps \Theta G$ defines a bijective coarse equivalence between the quotient coarse structure $\ku E_q\big(\Delta\backslash (G\times G)\big)$ and the Roelcke coarse structure $\ku E_\wedge(G)$.
\end{theorem}

\begin{proof}
Observe first that, for any $A\subseteq G$,
\maths{
L^{\Delta\backslash (G\times G)}_{A\times A}
&\;\;=\;\;   \COn{\big(\Delta(g_1,g_2),\Delta(f_1,f_2)\big)} {(g_1,g_2)\inv \Delta(f_1,f_2)\cap (A\times A)\neq \emptyset}\\
&\;\;=\;\;   \COn{\big(\Delta(g_1,g_2),\Delta(f_1,f_2)\big)} {g_1Af_1\inv \cap g_2Af_2\inv \neq\emptyset}\\
&\;\;=\;\;   \COn{\big(\Delta(g_1,g_2),\Delta(f_1,f_2)\big)} { g_1\inv g_2A\cap Af_1\inv f_2 \neq\emptyset}\\
&\;\;=\;\;   \COn{\big(\Delta(g_1,g_2),\Delta(f_1,f_2)\big)} { \Theta\big(\Delta(g_1,g_2)\big)\cdot A\;\cap \;A\cdot \Theta\big(\Delta(f_1,f_2)\big)\; \neq\;\emptyset}\\
&\;\;=\;\; (\Theta\times \Theta)\inv\big(L_A\circ R_A\big).
 }
Now, by \cite[Lemma 3.36]{MR4327092}, the bounded sets in $G\times G$ are those contained in products $A\times A$ of bounded subsets $A$ of $G$ and thus by Lemma \ref{ConcreteBasisQuotientCoarseStructure} the entourages $L^{\Delta\backslash (G\times G)}_{A\times A}$ are cofinal in the quotient structure $\ku E_q\big(\Delta\backslash (G\times G)\big)$. On the other hand, by  \eqref{eq:Roelcke basis}, the entourages $L_A\circ R_A$ are cofinal in $\ku E_\wedge(G)$ and hence the theorem follows from the equivalences above.
 \end{proof}

\begin{theorem}\label{thm:quotient-roelcke2}
Let $G$ be a topological group. Then the map $\Delta\backslash (G\times G) \maps \Theta G$ defines a bijective coarse equivalence between  $\ku E_{\rm Haus}\big(\Delta\backslash (G\times G)\big)$ and $\bigcap_d\ku E_{d_\wedge}(G)$.
\end{theorem}

\begin{proof}
Assume $d$ is a continuous left-invariant pseudometric on $G$ and let $(d\times d)\big((g_1,g_2),(f_1,f_2)\big)=\max\{d(g_1,f_1),d(g_2,f_2)\}$, which is continuous left-invariant on $G\times G$. Moreover, by left-invariance and the change of variable $k=g_1\inv h\inv f_2$, we find that
\maths{
(d\times d)_{\rm Haus}\big(\Delta(g_1,g_2),\Delta(f_1,f_2)\big)
& =\inf_{h\in G}\;(d\times d)\big((hg_1,hg_2),(f_1,f_2)\big)\\
& =\inf_{h\in G}\;\max\big\{d(hg_1,f_1),d(hg_2,f_2)\big\}\\
& =\inf_{h\in G}\;\max\big\{d(hg_1,f_1),d(g_2,h\inv f_2)\big\}\\
& =\inf_{k\in G}\;\max\big\{d(f_2k\inv ,f_1),d(g_2,g_1k)\big\}\\
& =\inf_{k\in G}\;\max\big\{d(k\inv ,f_2\inv f_1),d(g_1\inv g_2,k)\big\}\\
& =\inf_{k\in G}\;\max\big\{\check d(k ,f_1\inv f_2),d(g_1\inv g_2,k)\big\}\\
& =d_\wedge (g_1\inv g_2 ,f_1\inv f_2)\\
& =d_\wedge \Big(\Theta\big( \Delta(g_1,g_2)\big),\Theta\big(\Delta(f_1,f_2)\big)\Big).
} 
 Thus, $\Theta$ is an isometry between the two pseudometrics $(d\times d)_{\rm Haus}$ and $d_\wedge$. 
 
 It thus suffices to note that every continuous left-invariant pseudometric $D$ on $G\times G$ is bounded by some pseudometric of the form $d\times d$, which implies that
 $$
 \ku E_{\rm Haus}\big(\Delta\backslash (G\times G)\big)\;=\; \bigcap_d\ku E_{(d\times d)_{\rm Haus}}
 $$
 and therefore that $\Theta$ is a  coarse equivalence between $ \ku E_{\rm Haus}\big(\Delta\backslash (G\times G)\big)$ and $\bigcap_d\ku E_{d_\wedge}(G)$.
 \end{proof}

\begin{corollary}
For any topological group $G$ we have that 
$$
\mathcal{E}_{\wedge}(G)\;\subseteq\;\bigcap_d\ku E_{d_\wedge}(G).
$$ 
\end{corollary}

 \begin{proof}
 We have $\ku E_q\big(\Delta\backslash (G\times G)\big) \;\subseteq \; \ku E_{\rm Haus}\big(\Delta\backslash (G\times G)\big)$ and the bijection $\Theta\times \Theta$ maps $\ku E_q\big(\Delta\backslash (G\times G)\big)$ onto $\ku E_\wedge(G)$, while mapping $ \ku E_{\rm Haus}\big(\Delta\backslash (G\times G)\big)$ onto $\bigcap_d\ku E_{d_\wedge}(G)$.
 \end{proof}

 Considering the Roelcke coarse structure as an operator $\mathcal{E}_{\wedge}$, we can employ the basis given by Equation (\ref{eq:Roelcke basis}) to see it as a functor from the categories of topological groups to the category of coarse spaces. Indeed, if $G\maps\phi F$ is a continuous  group homomorphism between  topological groups, then for any  bounded subset $A\subseteq G$ we have that
\begin{equation}\label{eq:RoelckeFunctorial}
  (\phi\times \phi)[L_{A}\circ R_{A}]\subseteq L_{\phi[A]}\circ R_{\phi[A]},  
\end{equation}
which shows that $\phi$ is a controlled map between the Roelcke coarse structures on $G$ and $F$.

We also have that the Roelcke coarse structure preserves products. By \cite[Lemma 3.36]{MR4327092} we have that for a family of topological groups $\{G_{i}\}_{i\in I}$ the group bornology of bounded subsets of its product $G:=\prod_{i\in I}G_{i}$ has a basis made of arbitrary products $A=\prod_{i\in I}A_{i}$ of bounded subsets $A_{i}\subseteq G_{i}$. Consequently, for such products of the form $A=\prod_{i\in I}A_{i}$ we have that

\begin{equation}\label{eq:RoelckeProduct}
L_{A}\circ R_{A}=\bigg(\prod L_{A_i}\bigg)\circ\bigg(\prod R_{A_i}\bigg)=\prod_{i\in I}(L_{A_i} \circ R_{A_i}).    
\end{equation}
As the sets of the left-hand side of Equation \ref{eq:RoelckeProduct} are a basis for the Roelcke coarse structure on $G$ and the sets of the right-hand are a basis for the product coarse structure on $G$ coming from the Roelcke coarse structures on the $G_{i}$, we have that both coarse structures on $G$ coincide.

As expected, the question of whether the functor \(\mathcal{E}_{\land }\) preserves quotients yields both positive and negative results. We will prove that $\mathcal{E}_{\wedge}$ preserves quotients exactly when the functor $\mathcal{E}_{L}$ does, that is, when every bounded subset $B$ of a quotient topological group $G/H$ can be lifted to a bounded subset $A$ of the topological group $G$. 

Again, let $G$ be a topological group, $H$ a closed normal subgroup and $G\maps \pi G/H$ the quotient map. By  $\big(\mathcal{E}_{\wedge}\big)_q(G/H)$ we will denote the quotient coarse structure on $G/H$ associated to $\mathcal{E}_{\wedge}(G)$. We claim that the latter  is generated by the family of entourages
\begin{align}\label{GeneratingFamilyCoarseRoelcke}
\Con{L_{\pi[A]}\circ R_{\pi[A]}} {A\subseteq G\text{ is a bounded subset}}.    
\end{align}
Indeed, the quotient coarse structure $\big(\mathcal{E}_{\wedge}\big)_q(G/H)$ is generated by the images $(\pi\times \pi)[L_{A}\circ R_{A}]$ of  entourages $L_{A}\circ R_{A}\in \ku E_\wedge (G)$ and so, by Equation (\ref{eq:RoelckeFunctorial}), the entourages $L_{\pi[A]}\circ R_{\pi[A]}$ generate a coarse structure containing $\big(\mathcal{E}_{\wedge}\big)_q(G/H)$.
Conversely, for any bounded subset $A\subseteq G$, we have  $L_{A}\in \mathcal{E}_{\wedge}(G)$ and so $(\pi\times \pi)[L_A]=L_{\pi[A]}\in \big(\mathcal{E}_{\wedge}\big)_q(G/H)$. Similarly, $(\pi\times \pi)[R_A]=R_{\pi[A]}\in \big(\mathcal{E}_{\wedge}\big)_q(G/H)$ and hence $L_{\pi[A]}\circ R_{\pi[A]} \in \big(\mathcal{E}_{\wedge}\big)_q(G/H)$.

\begin{proposition}
    Let $H$ be a closed normal subgroup of a topological group $G$. Then 
    $$
    \big(\mathcal{E}_{\wedge}\big)_q(G/H)\;\subseteq\; \mathcal{E}_{\wedge}(G/H)
    $$ 
    and equality holds exactly when  bounded subsets of $G/H$ can be lifted to  bounded subsets of $G$.
\end{proposition}

\begin{proof}
The inclusion from left to right follows from $\pi$ being controlled between the Roelcke coarse structures on $G$ and $G/H$.

    If the lifting of bounded subsets from $G/H$ to $G$ can always be made, then the two coarse structures $\big(\mathcal{E}_{\wedge}\big)_q(G/H)$ and $\mathcal{E}_{\wedge}(G/H)$ share a common basis, and thus are equal.

    Conversely, suppose that $\big(\mathcal{E}_{\wedge}\big)_q(G/H)=\mathcal{E}_{\wedge}(G/H)$ and consider any bounded subset $B\subseteq G/H$ containing the identity $H=1_{G/ H}\in G/ H$. In order to lift $B$ to a bounded subset of $G$, we first use the supposed equality to imply the existence of a bounded subset $A\subseteq G$ for which
    $$
    L_{B}\circ R_{B}\subseteq L_{\pi[A]}\circ R_{\pi[A]}.
    $$
We claim that $B\subseteq \pi[AA\inv]$. Indeed, if $\xi\in B$, then, as $(H,H)\in R_{B}$, we have 
$$
(\xi^{-1},H)\;\in\; L_{B}\;\subseteq\; L_{B}\circ R_{B} \;\subseteq \; L_{\pi[A]}\circ R_{\pi[A]}
$$
and so we may find an $\eta\in G/H$ for which $(\xi^{-1},\eta)\in L_{\pi[A]}$ and $(\eta,H)\in R_{\pi[A]}$. Therefore, $\eta,\xi\eta\in \pi[A]$ and  $\xi\in \pi[AA\inv]$. Thus $B\subseteq \pi[AA\inv]$.
\end{proof}


\section{The Roelcke coarse structure on  homeomorphism groups}
Let \(S\) be a compact connected orientable surface without boundary and $\widetilde S\maps p S$
be the universal covering. Equip $\widetilde S$ with a compatible proper metric \(d\) invariant under deck transformations and let $d_\infty$ be the metric on 
\(\widetilde S\times\widetilde S\) given by 
\[
d_\infty\bigl((x,y),(x',y')\bigr)
=
\max\big\{d(x,x'),d(y,y')\big\}.
\]

Suppose $f,g\in {\sf Homeo}_0(S)$ and $\tilde f, \tilde g\in {\sf Homeo}_0(\widetilde S)$ are lifts of $f$ and $g$ respectively. Then any lift $\tilde a\in {\sf Homeo}_0(\widetilde S)$ of a homeomorphism $a\in {\sf Homeo}_0(S)$ defines a homeomorphism $\mathcal G{\tilde f}\maps {\varphi_{\tilde a}}\mathcal G{\tilde g}$ between the two graphs by the formula
$$
{\varphi_{\tilde a}}\big(z,\tilde f(z)\big)=\big(\tilde a(z),\tilde g\tilde a(z)\big).
$$
We say that a homeomorphism $\mathcal G{\tilde f}\maps {\varphi}\mathcal G{\tilde g}$ is {\em admissible} for $\tilde f$ and $\tilde g$ if $\varphi=\varphi_{\tilde a}$ for some such lift $\tilde a$. In general, if lifts $\tilde f$ and $\tilde g$ are fixed and $\mathcal G{\tilde f}\maps {\varphi}\mathcal G{\tilde g}$ is some homeomorphism, then since each coordinate projection identifies the graphs with $\widetilde S$, there is a unique homeomorphism $\vartheta\in {\sf Homeo}(\widetilde S)$ such that 
$$
\varphi\big(z,\tilde f(z)\big)=\big(\vartheta(z), \tilde g\vartheta(z)\big).
$$
Thus, $\varphi$ is admissible for $\tilde f$ and $\tilde g$ provided that this $\vartheta$ is a lift of some homeomorphism $a\in {\sf Homeo}_0(S)$. We now define a metric $\rho$ on ${\sf Homeo}_0(S)$ by the formula
\[
\rho(f,g)
=
\inf
\Bigg\{
\sup_{\zeta\in\mathcal G{\tilde f}}
d_\infty\big(\varphi(\zeta),\zeta\big)
\;\;\Bigg|\;\;
\begin{array}{l}
\mathcal G{\tilde f}\maps {\varphi}\mathcal G{\tilde g}\text{ is admissible for some lifts }
\tilde f, \tilde g\text{ of $f$ and $g$}
\end{array}
\Bigg\}.
\]

\begin{theorem}
The metric \(\rho\) generates the Roelcke coarse structure on
\(
{\sf Homeo}_0(S)
\).
\end{theorem}

\begin{proof}
For \(h\in {\sf Homeo}_0(S)\), define its minimal displacement on $\widetilde S$ by the formula
\[
\Delta(h)
=
\inf_{\tilde h}\;
\sup_{z\in\widetilde S}\;
d\big(\tilde h(z),z\big),
\]
where the infimum is taken over all lifts \(\tilde h\) of \(h\). Note that $\Delta(hg)\leqslant \Delta(h)+\Delta(g)$.
We claim that a subset \(B\subseteq {\sf Homeo}_0(S)\) is $\ku E_L$-bounded if and only if \(\Delta\) is bounded on \(B\). Indeed, by \cite[Theorem 1]{MR3846725}, the left coarse structure on  \(\operatorname{Homeo}_0(S)\) is given by the fragmentation metric and work of E. Militon \cite{MR3159168} identifies the fragmentation metric with displacement on the universal cover up to quasi-isometry.

As established earlier, the Roelcke coarse structure on \({\sf Homeo}_0(S)\) is generated by the entourages
\[
L_B\circ R_B=\Con{(f,g)\in {\sf Homeo}_0(S)\times {\sf Homeo}_0(S)}{ gB\cap Bf\neq\emptyset},
\]
where \(B\) ranges over the left-coarsely bounded subsets of \({\sf Homeo}_0(S)\). It thus follows that
\[
\rho'(f,g)
=
\inf\Con{\max\{\Delta(a),\Delta(b)\}}
{ga=b f,\; a,b\in {\sf Homeo}_0(S)}
\]
is a compatible metric for the Roelcke coarse structure.

We claim that $\rho\leqslant \rho'$. So assume $f,g, a,b\in {\sf Homeo}_0(S)$ are given such that $ga=bf$ and that $\tilde a$ and $\tilde b$ are lifts of $a$ and $b$ respectively. Let $\tilde f$ be any lift of $f$ and set $\tilde g=\tilde b \tilde f\tilde a\inv$, which is a lift of $g$. Let  $\mathcal G{\tilde f}\maps {\varphi_{\tilde a}}\mathcal G{\tilde g}$ be defined as above and note that
\maths{
\rho(f,g)
&\;\leqslant \;
\sup_{\zeta\in \mathcal G{\tilde f}}  d_\infty\big(\varphi_{\tilde a}(\zeta),\zeta\big)\\
&\;= \;
\sup_{z\in \widetilde S}  d_\infty\Big( \big(\tilde a(z),\tilde g\tilde a(z)\big), \big(z,\tilde f(z)\big)\Big)\\
&\;= \;
\sup_{z\in \widetilde S}  d_\infty\Big( \big(\tilde a(z),\tilde b\tilde f(z)\big), \big(z,\tilde f(z)\big)\Big)\\
&\;=\;
\sup_{z\in \widetilde S}\max\Big\{ d\big(\tilde a(z),z\big), d\big(\tilde b\tilde f(z),\tilde f(z)\big)\Big\}\\
&\;= \;
\max\Big\{ \sup_{z\in \widetilde S}d\big(\tilde a(z),z\big), \sup_{u\in \widetilde S}d\big(\tilde b(u),u\big)\Big\}.
}
By taking the infimum over all $a$ and $b$ along with their lifts $\tilde a$ and $\tilde b$, we find that $\rho(f,g)\leqslant \rho'(f,g)$.

We now establish the converse claim that $\rho'\leqslant \rho$. For this, assume $f,g, a\in {\sf Homeo}_0(S)$ are given and that $\tilde f$, $\tilde g$ and $\tilde a$ are lifts. We let $b=gaf\inv$ and set $\tilde b=\tilde g\tilde a\tilde f\inv$. We then have
\maths{
\rho'(f,g)
&\;\leqslant \;   \max\{\Delta(a),\Delta(b)\} \\
&\;\leqslant \;\max\Big\{ \sup_{z\in \widetilde S}d\big(\tilde a(z),z\big), \sup_{u\in \widetilde S}d\big(\tilde b(u),u\big)\Big\}\\
&\;=\;\max\Big\{ \sup_{z\in \widetilde S}d\big(\tilde a(z),z\big), \sup_{u\in \widetilde S}d\big(\tilde g\tilde a\tilde f\inv(u),u\big)\Big\}\\
&\;=\;\max\Big\{ \sup_{z\in \widetilde S}d\big(\tilde a(z),z\big), \sup_{z\in \widetilde S}d\big(\tilde g\tilde a(z),\tilde f (z)\big)\Big\}\\
&\;= \;\sup_{z\in \widetilde S}\; \max\Big\{ d\big(\tilde a(z),z\big), d\big(\tilde g\tilde a(z),\tilde f (z)\big)\Big\}\\
&\;= \;
\sup_{z\in \widetilde S} d_\infty\Big( \big(\tilde a(z),\tilde g\tilde a(z)\big), \big(z,\tilde f(z)\big)\Big)\\
&\;= \;
\sup_{\zeta\in \mathcal G{\tilde f} } d_\infty\big(\varphi_{\tilde a}(\zeta),\zeta\big).
}
Taking the infimum over $\tilde f$, $\tilde g$, $a$ and $\tilde a$, we find that $\rho'\leqslant \rho$ and so $\rho=\rho'$.
\end{proof}


\section{Positive results}
Our first positive result provides a complete characterization of when the left and quotient coarse structures coincide on $G/H$ provided the latter group is locally bounded.  For this, let us recall that a map $X\maps\phi Y$ between two coarse spaces is said to be {\em modest} if $\phi[A]$ is bounded in $Y$ whenever $A$ is a bounded subset of $X$. The specific setting that interests us is when $H$ is a closed normal subgroup of a topological group $G$ and $G\maps\pi G/H$ is the quotient map. In this case, a {\em section} $G/H\maps \phi G$ for $\pi$, that is a function such that $\pi\circ\phi={\sf id}_{G/H}$, will be said to be {\em modest} if it is modest as a map
$$
\big(G/H,\ku E_L(G/H)\big)\maps \phi \big(G,\ku E_L(G)\big).
$$
Of course, the existence of a modest section is generally stronger than every $\ku E_L(G/H)$-bounded set being the image of a bounded set from $G$, which by Corollary \ref{cor:lifting crit} itself is equivalent to the equality of the quotient and left coarse structures on $G/H$. We also remark that, if $G$ is a Polish group and $G/H\maps \phi G$ is a modest section for $\pi$, then one can find a C-measurable map $G/H\maps \psi G$ with $\ku G\psi\subseteq \overline{\ku G \phi}$ and this $\psi$ remains a modest section for $\pi$ \cite[Proposition 3.28]{MR4327092}. Here the algebra of $C$ sets in a Polish space is the smallest $\sigma$-algebra containing the Borel sets and closed under the Souslin operation $\ku A$.  Thus any modest section can be upgraded to be reasonably measurable.

\begin{theorem}
Suppose $H$ is a closed normal subgroup of a Polish group $G$ and assume that the quotient group $G/H$ is locally bounded or, equivalently, that $\ku E_L(G/H)$ is metrisable. Then the following are equivalent.
\begin{enumerate}
\item $\mathcal E_q(G/H)=\mathcal E_L(G/H)=\mathcal E_{\rm Haus}(G/H)$,
\item every $\mathcal E_L(G/H)$-bounded set $B$ is of the form $\pi[A]$ for some bounded $A\subseteq G$,
\item $H$ is  {\em $\sigma$-cobounded} in $G$, that is, $G=\bigcup_{n=1}^\infty HA_n$ for a sequence of bounded subsets $A_n\subseteq G$, 
\item there is a modest section $G/H\maps\phi G$ for the quotient map $\pi$.
\end{enumerate}
\end{theorem}

\begin{proof}
Because $G/H$ is a locally bounded Polish group, it can be covered by a sequence $B_1\subseteq B_2\subseteq \ldots$ of $\mathcal E_L(G/H)$-bounded open subsets satisfying $B_n^2\subseteq B_{n+1}$ for all $n$. We fix this sequence and also let $B_0=\emptyset$. Then, by \cite[Proposition 2.15]{MR4327092}, a set $B$ is $\mathcal E_L(G/H)$-bounded  if and only if it is contained in some $B_n$. Observe that the equivalence of (1) and (2) is immediate from Corollary \ref{cor:lifting crit}.

(3)$\Rightarrow$(2):  Assume $G=\bigcup_{n=1}^\infty HA_n$ for a sequence of bounded subsets $A_n\subseteq G$. By the Baire category, some set $HA_n$ must be nonmeagre in $G$ and so the larger analytic set $H\overline{A_n}$ must be somewhere comeagre. Applying Pettis' lemma \cite[Lemma 2.1]{MR2535429}, we see that $H\overline{A_n}H\overline{A_n}=H\overline{A_n}^2$ must have nonempty interior and so, as $\pi$ is an open mapping,  $\pi\big[H\overline{A_n}^2\big]=\pi\big[\overline{A_n}^2\big]$ similarly has nonempty interior. Therefore, if $B\subseteq G/H$ is any $\ku E_L(G/H)$-bounded set, there is some finite set $E\subseteq G/H$ and some number $p$ so that
$B\subseteq \Big( E\pi\big[\overline{A_n}^2\big]\Big)^p$.
Writing $E=\pi[F]$ for some finite set $F\subseteq G$, we see that 
$$
B\subseteq \Big( E\pi\big[\overline{A_n}^2\big]\Big)^p=\Big( \pi[F]\pi\big[\overline{A_n}^2\big]\Big)^p=\pi\Big[\big(F\overline{A_n}^2\big)^p\Big].
$$
Thus, any $\ku E_L(G/H)$-bounded set is the image of a bounded set from $G$.

(4)$\Rightarrow$(3):  Suppose $G/H\maps\phi G$ is  a modest section for $\pi$. Then $A_n=\phi[B_n]$ is bounded in $G$ and $G=\bigcup_{n=1}^\infty HA_n$.

(2)$\Rightarrow$(4): Assume that (2) holds and find bounded sets $A_n\subseteq G$ so that $\pi[A_n]=B_n$. For every $x\in B_{n+1}\setminus B_n$, let $\phi(x)\in A_{n+1}\cap\pi\inv(x)$ be any element. Then, if $B$ is $\ku E_L(G/H)$-bounded, we have $B\subseteq B_n$ for some $n$ and so $\phi[B]\subseteq A_1\cup\cdots\cup A_n$, whereby $\phi[B]$ is bounded in $G$. So $\phi$ is modest. 
\end{proof}

\begin{theorem}\label{SecondMainCriterion}
Let $H$ be a closed normal subgroup of a Polish group $G$. Suppose there is an analytic transversal $T\subseteq G$ of the quotient map $G\maps \pi G/H$ for which the intersection $(T^{-2}T)\cap H$ is bounded in $G$. Then the section $G/H\maps \phi G$ for $\pi$ defined by $T$ is Borel measurable and modest with respect to the left coarse structures on $G/H$ and $G$. It follows that 
$$
\mathcal E_q(G/H)=\mathcal E_L(G/H)=\mathcal E_{\rm Haus}(G/H).
$$
\end{theorem}

\begin{proof}
To see that $G/H\maps \phi G$ is Borel measurable, by \cite[Theorem 14.12]{MR1321597}, it is enough to show that it has analytic graph
$$
\ku G\phi=\COn{\big(\pi(g),g\big)\in G/H\times G}{g\in T},
$$
which is immediate because $T$ is analytic. 
Observe also that, for all $x,y\in G/H$, we have 
$$
\phi(y)\inv\phi(x)\inv \phi(xy)\in (T^{-2}T)\cap \ker(\pi)=(T^{-2}T)\cap H
$$ 
and thus $\phi(xy)\in \phi(x)\phi(y)(T^{-2}T\cap H)$.

Next, to see that $\phi$ is modest, let $B\subseteq G/H$ be any given $\mathcal{E}_{L}(G/H)$-bounded set. We must show that also $\phi[B]$ is bounded in $G$. So let $V\subseteq G$ be an identity neighbourhood. Since $G$ is separable, pick a countable dense subset $\{c_{n}\}_{n\in\mathbb{N}}\subseteq G$, whence $G=\bigcup_{n\in \mathbb{N}}c_{n}V$ and so
$$
G/H=\bigcup_{n\in \mathbb{N}}\phi^{-1}(c_nV).
$$
By the Baire category theorem, there is an $n\in \mathbb{N}$ for which the set $A=\phi^{-1}(c_nV)$ is nonmeagre in $G/H$. Consequently, since $A$ also has the Baire property, we have that $A$ is somewhere comeagre. Also, by  Pettis' lemma \cite[Lemma 2.1]{MR2535429},  $A^2$ has nonempty interior and so, because $B$ is bounded, we can find a finite subset $F\subseteq G/H$ and a natural number $q$ for which $B\subseteq (FA)^{q}$. However, for all $f_i\in F$ and $a_i\in A$, we find that
\begin{align*}
  \phi(f_1a_1f_2a_2\,\cdots\, f_qa_q)
  &\;\in\;\phi(f_1)\phi(a_{1}f_2a_2\,\cdots\, f_{q}a_{q})(T^{-2}T\cap H)\\
  &\;\subseteq\;\phi(f_1)\phi(a_{1})\phi(f_2a_2\,\cdots\, f_{q}a_{q})(T^{-2}T\cap H)^{2}\\
  &\quad\vdots\\
  &\;\subseteq\; \phi(f_1)\phi(a_1)\cdots\phi(f_q)\phi(a_q)(T^{-2}T\cap H)^{2q},
\end{align*}
showing that $\phi[B]\subseteq \big(\phi[F](c_nV)\big)^{q}\big(T^{-2}T\cap H\big)^{2q}$.
Since, by hypothesis, the intersection $T^{-2}T\cap H$ is bounded in $G$, there is a finite subset $F'\subseteq G$ and an $m\in \mathbb{N}$ for which $(T^{-2}T\cap H)^{2q}\subseteq (F'V)^{m}$. Consequently, setting $F''=\phi[F]\cup\{c_n\}\cup F'$, we conclude that $\phi[B]\subseteq (F''V)^{2q+m}$. Because $V$ was arbitrary, this shows that $\phi[B]$ is bounded.
\end{proof}

\begin{proposition}\label{prop:FirstMainCriterion}
Suppose $H$ is a closed subgroup of a topological group $G$ for which there is a continuous left-invariant pseudometric $\partial$ on $G$ such that 
$$
\ku E_\partial(G)\upharpoonright H \;=\; \ku E_L(G)\upharpoonright H
$$
and let $G\maps \pi H\backslash G$ denote the quotient map. Then a subset $A\subseteq G$ is bounded if and only if 
${\sf diam}_\partial(A)<\infty$ and $\pi[A]$ is $\ku E_{\rm Haus}(H\backslash G)$-bounded. It follows that every $\ku E_{\rm Haus}(H\backslash G)$-bounded set $B$ is of the form $B=\pi[A]$ for some bounded subset $A\subseteq G$.
\end{proposition}

\begin{proof}
If $A\subseteq G$ is bounded, i.e., is $\ku E_L(G)=\bigcap_d\ku E_d(G)$ bounded,  then it in particular has finite $\partial$-diameter. Moreover, $\pi[A]$ will be $\ku E_q(H\backslash G)$-bounded by the definition of the quotient coarse structure and hence also bounded with respect to the larger coarse structure $\ku E_{\rm Haus}(H\backslash G)$.

Conversely, suppose $A\subseteq G$ has finite $\partial$-diameter and $\pi[A]$ is $\ku E_{\rm Haus}(H\backslash G)$-bounded. Assume for a contradiction that $A$ is unbounded in $G$ and find some continuous left-invariant pseudometric $d$ on $G$ so that ${\sf diam}_d(A)=\infty$. Without loss of generality, we may assume that $\partial \leqslant d$.  Choose $a_n\in A$ so that $\lim_nd(a_n,a_1)=\infty$  and note that, because $\pi[A]$ is $\ku E_{\rm Haus}(H\backslash G)$-bounded, 
we have $K=\sup_nd_{\rm Haus}(H,Ha_n)<\infty$. It follows that there are $h_n\in H$ so that $\partial(h_n,a_n)\leqslant d(a_n,h_n)<K+1$ for all $n$ and so 
$$
\sup_n\partial(h_n,a_1)\leqslant \sup_n\partial(h_n,a_n)+ \sup_n\partial(a_n,a_1)\leqslant K+1+ {\sf diam}_\partial(A)<\infty.
$$
That is, the set $\{h_n\}_n$ is a $\partial$-bounded subset of $H$ and must therefore be bounded in $G$ by our assumptions on $\partial$. In particular, it is $d$-bounded and so, as $\sup_nd(a_n,h_n)\leqslant K+1$, we find that also $\{a_n\}_n$ is $d$-bounded, contradicting the choice of the $a_n$.

Suppose now $B\subseteq H\backslash G$ is $\ku E_{\rm Haus}(H\backslash G)$-bounded. Define
\maths{
A
&=\COn{g\in G}{Hg\in B\quad\&\quad \partial(g,1)\leqslant \partial_{\rm Haus}\big(Hg,H\big)+1\,}.
}
Clearly $\pi[A]=B$. So to see that $A$ is bounded, it suffices to see that it has finite $\partial$-diameter. But $B$ is $\ku E_{\rm Haus}(H\backslash G)$-bounded and thus $\partial_{\rm Haus}$-bounded, whence 
$$
\partial(g,1)\;\leqslant\; \partial_{\rm Haus}\big(Hg,H\big)+1\;\leqslant\; \sup_{x\in B}\partial_{\rm Haus}(x,H)+1\;<\;\infty
$$
for all $g\in A$. Thus ${\sf diam}_\partial(A)<\infty$ as required.
\end{proof}

When $H$ is moreover normal in $G$,  we can conclude actual coincidence of coarse structures on the quotient $G/H$. 

\begin{theorem}\label{FirstMainCriterion}
Suppose $H$ is a closed normal subgroup of a topological group $G$ for which there is a continuous left-invariant pseudometric $\partial$ on $G$ such that 
$$
\ku E_\partial(G)\upharpoonright H \;=\; \ku E_L(G)\upharpoonright H.
$$
Then 
$$
\mathcal E_q(G/H)=\mathcal E_L(G/H)=\mathcal E_{\rm Haus}(G/H).
$$
\end{theorem}

 \begin{proof}
 This follows directly from Corollary \ref{cor:lifting crit} and Proposition \ref{prop:FirstMainCriterion}.
 \end{proof}
 
The utility of Theorem \ref{FirstMainCriterion} is somewhat restricted by the fact that the pseudometric $\partial$ must be defined on all of $G$ and not only on the subgroup $H$. We do not know if such a weakening suffices in all cases, but we shall show that it does for the case when $G$ is a countable direct product of locally compact groups. This relies on  a fact of independent interest.

\begin{lemma}[Compactness criterion]\label{Compactness}
Let $\{\mathcal{E}_{n}\}_{n\in\mathbb{N}}$ be a sequence of connected coarse structures on a set $X$ for which the intersection $\bigcap_n\mathcal E_n$ is metrisable. Then, if  $E_{n}\in\mathcal{E}_{n}$ for all $n$, we have  
$$
E_{1}\cap\cdots\cap E_{k}\in\bigcap_{n\in\mathbb{N}}\mathcal{E}_{n}
$$
for some $k$.
It follows that if $B_n\subseteq X$ is $\mathcal E_n$-bounded for each $n$, then  
$$
B_1\cap \cdots \cap B_k
$$
is $\bigcap_n\mathcal E_n$-bounded for some $k$.
\end{lemma}

\begin{proof}
 Let $\rho$ be any metric on $X$ metrising the coarse structure $\bigcap_{n\in\mathbb{N}}\mathcal{E}_{n}$ and suppose for a contradiction that $E_n\in \ku E_n$ for all $n$,  but that $E_1\cap \cdots\cap E_k\notin \bigcap_{n\in \mathbb N}\ku E_n$ for all $k$. Then, for every $k$, there is some  $(x_k,y_k)\in E_{1}\cap\cdots\cap E_{k}$ with $\rho(x_k,y_k)>k$. It follows that 
 $$
\{(x_k,y_k)\}_{k\in\mathbb{N}}\;\notin\; \mathcal{E}_{\rho}\;=\;\bigcap_{n\in\mathbb{N}}\mathcal{E}_{n}
 $$ 
and so  $\{(x_k,y_k)\}_{k\in\mathbb{N}}\not\in \mathcal{E}_{n}$ for some $n$. However, by construction, $\{(x_k,y_k)\}_{k\geqslant n} \subseteq E_n \in \ku E_n$ and, as $\ku E_n$ is connected, we have $\{(x_1,y_1),\cdots,(x_{n-1},y_{n-1})\}\in \ku E_n$. It therefore follows that
$$
\{(x_k,y_k)\}_{k\in\mathbb{N}}\;=\; \{(x_1,y_1),\cdots,(x_{n-1},y_{n-1})\}\;\cup\; \{(x_k,y_k)\}_{k\geqslant n} \in \ku E_n,
$$
contradicting our assumptions.
\end{proof}

\begin{theorem}\label{thm:prod comp}
Let $G=\prod_{n=1}^\infty G_n$ be a product of locally compact, $\sigma$-compact groups $G_n$ and let $H$ be a closed normal subgroup of $G$. If $\ku E_L(G)\upharpoonright H$ is metrisable, then 
$$
\mathcal E_q(G/H)=\mathcal E_L(G/H)=\mathcal E_{\rm Haus}(G/H).
$$
\end{theorem}

\begin{proof}
For simplicity of notation, we shall view all factors $G_n$ and subproducts $\prod_{n\in A}G_n$ as actual subgroups of $G$. Let \(G\maps{{\sf proj}_A}\prod_{n\in A}G_n\) denote the canonical projections. 

By Kakutani--Kodaira \cite{MR0014401} and Struble \cite{MR0348037}, there are continuous proper left-invariant pseudometrics $d_n$ on $G_n$. Indeed, choose a compact normal subgroup
$N\lhd G_n$ such that $G_n/N$ is metrisable, equip $G_n/N$ with a proper left-invariant compatible metric, and pull this metric back to $G_n$. It thus follows that $\ku E_L(G_n)=\ku E_{d_n}(G_n)$ and so by \cite[Section 3.5]{MR4327092}
$$
\ku E_L\big(G)=\ku E_L\big(\prod_{n=1}^\infty G_n)=\prod_{n=1}^\infty \ku E_L(G_n)=\prod_{n=1}^\infty \ku E_{d_n}(G_n).
$$
Here, $E\in \prod_{n=1}^\infty \ku E_{d_n}(G_n)$ if and only if $\big({\sf proj}_{\{n\}}\times {\sf proj}_{\{n\}}\big)[E]\in \ku E_{d_n}(G_n)$ for all $n$, that is, if $\sup_{(g,f)\in E}\;d_n\big({\sf proj}_{\{n\}}(g),{\sf proj}_{\{n\}}(f)\big)<\infty$ for all $n$.
For $g,f\in G$, let
$$
\partial_n(g,f)=d_n\big({\sf proj}_{\{n\}}(g),{\sf proj}_{\{n\}}(f)\big).
$$
So the $\partial_n$ are continuous left-invariant pseudometrics on $G$ and 
$$
\ku E_L\big(G)=\prod_{n=1}^\infty \ku E_{d_n}(G_n)=\bigcap_{n=1}^\infty \ku E_{\partial_n}(G).
$$

For each $n$, choose a compact identity neighbourhood $V_n\subseteq G_n$ and set
\[
C_n=H\cap {\sf proj}_{\{n\}}^{-1}(V_n),
\]
which is $\partial_n$-bounded. Since
$\mathcal E_L(G)\upharpoonright H$ is metrisable, Lemma \ref{Compactness} gives
$k$ such that
\[
C=C_1\cap\cdots\cap C_k=H\cap {\sf proj}_{\{1,\ldots, k\}}^{-1}(V_1\cdots V_k)
\]
is $\mathcal E_L(G)$-bounded. So let $V=V_1\cdots V_k$, which is a compact identity neighbourhood in the locally compact group $\prod_{n=1}^kG_n$.
Note that ${\sf proj}_{\{1,\ldots, k\}}[C]=V\cap {\sf proj}_{\{1,\ldots, k\}}[H]$ and put $\partial=\partial_1+\cdots+\partial_k$.

We claim that $\ku E_\partial(G)\upharpoonright H=\ku E_L(G)\upharpoonright H$, which by Theorem \ref{FirstMainCriterion} implies our result. For this, it suffices to check that every \(\partial\)-bounded subset \(B\subseteq H\) is \(\mathcal{E}_L(G)\)-bounded.  Since \(B\) is \(\partial\)-bounded and the pseudometrics
$d_1,\ldots,d_k$ are proper, ${\sf proj}_{\{1,\ldots, k\}}[B]$ has compact closure in $\prod_{n=1}^kG_n$. This implies that there is a finite set \(E\subseteq B\) such that
\maths{
{\sf proj}_{\{1,\ldots, k\}}[B]
&\subseteq {\sf proj}_{\{1,\ldots, k\}}[E]\cdot V.
}
However, as $E\subseteq B\subseteq H$, this implies that
\maths{
{\sf proj}_{\{1,\ldots, k\}}[B]
&\subseteq \Big({\sf proj}_{\{1,\ldots, k\}}[E]\cdot V \Big)\cap {\sf proj}_{\{1,\ldots, k\}}[H]\\
&\subseteq {\sf proj}_{\{1,\ldots, k\}}[E]\cdot \Big(V \cap {\sf proj}_{\{1,\ldots, k\}}[H]\Big)\\
&\subseteq {\sf proj}_{\{1,\ldots, k\}}[E]\cdot {\sf proj}_{\{1,\ldots, k\}}[C]\\
&={\sf proj}_{\{1,\ldots, k\}}[EC].
}
Therefore,
$$
B\subseteq \Big(EC\cdot {\sf ker}\big({\sf proj}_{\{1,\ldots,k\}}\big)\Big)\cap H 
=EC\cdot\Big( {\sf ker}\big({\sf proj}_{\{1,\ldots,k\}}\big)\cap H\Big) \subseteq EC^2.
$$
Because \(C\) is \(E_L(G)\)-bounded and \(E\) is finite, \(B\) is \(E_L(G)\)-bounded, proving our claim and thus the theorem.
\end{proof}


\section{Another counterexample}
\label{sect:BonetDierolf}

We finish by recording that Question \ref{quest:1} has a negative answer already among abelian Polish groups.  This is a direct coarse-geometric reformulation of a standard example of J. Bonet and S. Dierolf \cite[Example~1]{BonetDierolf1993}. For this, let us begin by recalling that, if $X$ is a Banach space, then norm-distance on $X$ is a compatible metric for the left coarse structure of the underlying additive topological group.

Let $Y\subseteq X$ be separable Banach spaces such that the inclusion is continuous, dense and non-closed. For example, one could take $X=\ell_p$ and $Y=\ell_q$ for some $1\leqslant q<p<\infty$. Consider the additive topological group 
$$
G=c_0(X)\times Y^{\mathbb N},
$$
where $c_0(X)$ is the Banach space of null sequences in $X$ and $Y^{\mathbb N}$ is given the product topology. Define
$$
G\maps \pi X^{\mathbb N},\qquad \pi(x,y)=x+y,
$$
where $x\in c_0(X)\subseteq X^\mathbb N$ and $y\in Y^{\mathbb N}\subseteq X^{\mathbb N}$.
The map $\pi$ is a continuous epimorphism. Indeed, if $z=(z_n)\in X^{\mathbb N}$, then density of $Y$ in $X$ lets us choose $y_n\in Y$ with $\|z_n-y_n\|_X<\tfrac 1n$, whereby $x=(z_n-y_n)\in c_0(X)$ and  $\pi(x,y)=z$. By the open mapping theorem for Fr\'echet spaces or for Polish groups, $\pi$ is an open mapping. Therefore, $
H={\sf ker}\,\pi\leqslant G$ is a closed subgroup and  the quotient group $G/H$ is homeomorphically isomorphic with $X^{\mathbb N}$.

Let $B_X$ and $B_Y$ denote the closed unit balls in respectively $X$ and $Y$. Because $Y$ is dense, but not closed in $X$, the topology on $Y$ is strictly finer than the topology induced from the ambient space $X$. It follows  that 
$$
B_X\cap Y\not\subseteq \tfrac 12B_X + mB_Y
$$
for all $m\geqslant 1$. Set 
$$
C\;=\;\prod_{n=1}^\infty \big(B_X\cap Y\big)\;\subseteq \;X^{\mathbb N},
$$
which by \cite[Lemma 3.36]{MR4327092} is bounded in the quotient group $X^{\mathbb N}$. Assume for a contradiction that $C\subseteq\pi[D]$ for some bounded subset $D\subseteq G=c_0(X)\times Y^{\mathbb N}$. By \cite[Lemma 3.36]{MR4327092} once again, we may assume that $D\subseteq c_0(X)\times \prod_{n=1}^\infty m_nB_Y$ for some constants $m_n$. Pick then $y_n\in \big(B_X\cap Y\big)\setminus \big(\tfrac 12B_X + m_nB_Y\big)$ for each $n\geqslant 1$ and set $y=(y_n)\in C$. By assumption, we can write 
$$
y=\pi(x,z)
$$
for some $x\in c_0(X)$ and $z\in \prod_{n=1}^\infty m_nB_Y$. Because $x\in c_0(X)$, there is some $n$ large enough so that $x_n\in \tfrac12 B_X$, whereby $y_n=x_n+z_n\in \tfrac 12B_X+m_nB_Y$, contradicting our assumptions. 

This shows that the bounded set $C\subseteq \prod_{n=1}^\infty \big(B_X\cap Y\big)$ cannot be lifted to a bounded set in $G$ and therefore, by Corollary \ref{cor:lifting crit}, that $\ku E_q(G/H)\neq \ku E_L(G/H)=\prod_{n=1}^\infty \ku E_L(X)$.


\section{Conflict of interest and data availability statements}
The authors attest that there is no conflict of interest to report. No data is associated with the paper.


\bibliographystyle{alpha} 
\bibliography{bib}
\end{document}